\documentclass[10pt]{article}
\usepackage{amssymb, amsfonts, amsthm, mathrsfs, amsmath}

\theoremstyle{definition}
\newtheorem{thm}{Theorem}[section]

\newtheorem{lemma}[thm]{Lemma}

\newtheorem{ex}[thm]{Example}
\newtheorem{rmk}[thm]{Remark}

\def\ccc{\mathbb{C}}

\def\zz{\mathbb{Z}}
\def\rr{\mathbb{R}}

\def\p{\partial}

\def\ud{\mathrm{d}}

\def\bea{\begin{equation}}
\def\eea{\end{equation}}
\def\o{{\omega}}
\def\<{{\langle}}
\def\>{{\rangle}}
\def\na{{\nabla}}

\begin{document}

\title{Symplectic geometric flows}
\author{Teng Fei and Duong H. Phong}

\date{}

\maketitle{}

\begin{abstract}
Several geometric flows on symplectic manifolds are introduced which are potentially of interest in symplectic geometry and topology. They are motivated by the Type IIA flow and T-duality between flows in symplectic geometry and flows in complex geometry. Examples include the Hitchin gradient flow on symplectic manifolds, and a new flow which is called the dual Ricci flow.
\end{abstract}

\section{Introduction}

Geometric flows are now well-recognized as a powerful tool for geometry and topology. Major successes include the Eells-Sampson theorem on harmonic maps \cite{eells1964}, Hamilton's Ricci flow \cite{hamilton1982} and Perelman's proof of the Poincar\'e conjecture \cite{perelman2002, perelman2003, perelman2003b}, Donaldson's heat flow proof of the Donaldson-Uhlenbeck-Yau theorem on Hermitian-Yang-Mills connections \cite{donaldson1987}, the differentiable sphere theorem of Brendle-Schoen \cite{brendle2009}, and many other developments. However, these successes were mostly in the settings of Riemannian and complex geometry, leaving the subject of symplectic geometric flows underdeveloped. In recent years, a few geometric flows adapted to symplectic geometry have been introduced, such as L\^ e-Wang's anti-complexified Ricci flow \cite{le2001}, Streets-Tian's symplectic curvature flow \cite{streets2014}, and He's flow of non-degenerate $2$-forms \cite{he2021}. Although there has been much progress on these flows, applications of geometric flows in symplectic geometry and topology have remained relatively few.

\smallskip

In \cite{fei2020b}, in joint work with S. Picard and X.W. Zhang, we introduced the Type IIA flow for symplectic Calabi-Yau 6-manifolds, motivated by the Type IIA string equations proposed in e.g. \cite{tseng2014}. This was further developed in \cite{fei2020c, fei2021b, raffero2021}. In particular, in \cite{fei2021b}, we successfully applied the Type IIA flow to prove the stability of the K\"ahler property for Calabi-Yau 3-folds under symplectic deformations.

\smallskip

The goal of this paper is to introduce some new geometric flows which may potentially be of interest in symplectic geometry and topology. They are motivated by duality considerations as well as by the Type IIA flow, to which several of them are actually closely related. More specifically, in \cite{fei2021e}, S. Picard and the first author had introduced the dual Anomaly flow as the T-dual of the Type IIB flow.\footnote{This flow coincides with the special case of the Anomaly flow \cite{phong2018b, phong2018e} with slope parameter $\alpha'=0$, and  "Type IIB flow" is a more precise name, since there is no longer a gauge field and no Green-Schwarz Anomaly cancellation mechanism
\cite{phong2020, fei2020}.} Both the Type IIB flow and the dual Anomaly flow are flows in complex geometry, and they are only dual to each other up to lower order terms. Here we introduce instead
the principle that by applying T-duality, we can derive a flow in symplectic geometry from a flow in complex geometry. In particular, we show that the Type IIA flow and the Type IIB flow are indeed related to each other by T-duality, as suggested by mirror symmetry, and the duality is now exact without lower order terms. We also derive the T-dual of the K\"ahler Ricci flow in the symplectic setting. In Section 3, we propose the study of the gradient flow of Hitchin's functional in dimension six. In the presence of a compatible symplectic form, this flow takes a similar expression to that of the Type IIA flow. Moreover, it induces a flow of a pair which can be viewed as the anti-complexified Ricci flow with lower order corrections coupled to a scalar function. In Section 4 we present a few explicit examples of the gradient flow of 6D-Hitchin's functional, showing that the flow can be used to find optimal almost complex structures on certain locally homogeneous symplectic half-flat manifolds. These examples naturally provide eternal and convergent solutions to the anti-complexified Ricci flow.

\smallskip

\noindent\textbf{Acknowledgements} The authors dedicate this paper to the beloved Professor Victor Guillemin. Some ideas behind this paper date back to the many conversations the first-named author had with Victor when he was a graduate student at MIT. The second-named author would like to take this opportunity to express his gratitude for the inspiration he got from Professor Guillemin's works, especially in symplectic geometry and partial differential equations, and for the frequent encouragement he received from him over the years.

\section{The Type IIA and the Type IIB flows}

In this section, we would like to present a general principle by which a flow in symplectic geometry should arise by T-duality from a flow in complex geometry. We illustrate this principle by establishing the duality of the Type IIA and the Type IIB flows in the semi-flat case, which is the case when T-duality can be implemented explicitly.

\smallskip

Let us first recall the set-up of semi-flat geometry, in the notations from \cite{fei2021e}. Let $B$ be a 3-dimensional compact special integral affine manifold. We can cover $B$ by local coordinates $\{x^1,x^2,x^3\}$ such that the transition functions are valued in the group $\textrm{SL}(3,\zz)\ltimes\rr^3$. Let $TB$ be the tangent bundle of $B$, with $\{y^1,y^2,y^3\}$ the natural coordinates of tangent directions. Clearly the fiberwise lattice
\[\check\Lambda=\zz\frac{\p}{\p x^1}+\zz\frac{\p}{\p x^2}+\zz\frac{\p}{\p x^3}=\{(y^1,y^2,y^3):y^j\in\zz\}\]
is well-defined, and we may form the quotient $\check X=TB/\check\Lambda$, which is a smooth $T^3$-fibration over $B$. Moreover, $\check X$ is naturally equipped with a complex structure such that $\{z^j=x^j+iy^j\}_{j=1}^3$ forms a set of local holomorphic coordinates. For our purpose, we shall also assume that $B$ is equipped with a Hessian metric $g$ with local potential $\phi$. In other words, there exists a local convex function $\phi$ such that
\[g_{jk}=\frac{\p^2\phi}{\p x^j\p x^k}.\]

We shall also consider the dual affine coordinates $\left\{x_j=\dfrac{\p\phi}{\p x^j}\right\}$ on $B$ under Legendre transformation. Let $\{y_1,y_2,y_3\}$ be the natural coordinates for the cotangent directions associated to the local chart $\{x^1,x^2,x^3\}$. It is not hard to check that the fiberwise lattice
\[\hat\Lambda=\zz\ud x^1+\zz\ud x^2+\zz\ud x^3=\{(y_1,y_2,y_3):y_k\in\zz\}\]
is well-defined, and we may form the fiberwise quotient $\hat X=T^*B/\hat\Lambda$, which is also a smooth $T^3$-fibration over $B$. The natural symplectic form on $T^*B$ descends to a symplectic form on $\hat X$, which in local coordinates can be expressed as
\[\omega=\ud x^1\wedge\ud y_1+\ud x^2\wedge\ud y_2+\ud x^3\wedge\ud y_3.\]
Moreover, $\hat X$ is equipped with a holomorphic volume form $\Omega=\varphi+i\hat\varphi=\ud z_1\wedge\ud z_2\wedge\ud z_3$ for $z_j=x_j+iy_j$. In local coordinates we have
\[\begin{split}\varphi&=\ud x_1\wedge\ud x_2\wedge\ud x_3-\ud x_1\wedge\ud y_2\wedge\ud y_3-\ud y_1\wedge\ud x_2\wedge\ud y_3-\ud y_1\wedge\ud y_2\wedge\ud x_3,\\
\hat\varphi&=\ud x_1\wedge\ud x_2\wedge\ud y_3+\ud x_1\wedge\ud y_2\wedge\ud x_3+\ud y_1\wedge\ud x_2\wedge\ud x_3-\ud y_1\wedge\ud y_2\wedge\ud y_3.\end{split}\]

Now we fix the affine structure on $B$ associated to the local coordinate $\{x^1,x^2,x^3\}$ and vary the Hessian metric $g$ (or equivalently vary the potential $\phi$). The symplectic form $\omega$, coming from the canonical one on $T^*B$, does not change. However, the 3-forms $\varphi$ and $\hat\varphi$ are varying since the dual affine structure associated to $\{x_1,x_2,x_3\}$ is changing.

In particular, if we run the Type IIB flow on $\check X$, by the calculation in \cite{fei2021e}, the flow reduces to the real Monge-Amp\`ere flow of Hessian metrics
\begin{equation}
\p_t g_{jk}=\frac{1}{4}\frac{\p^2}{\p x^j\p x^k}\det g.\label{iib}
\end{equation}
The goal in this section is to prove the following:
\begin{thm}\label{mirror}~\\
Under the Type IIB flow (\ref{iib}), the associated 3-form $\varphi$ satisfies the Type IIA flow \cite{fei2020b}
\begin{equation}
\p_t\varphi=\frac{1}{16}\ud\Lambda_{\omega}\ud(|\varphi|^2\hat\varphi).
\end{equation}
\end{thm}
\begin{proof}
By the definition of Legendre transform, we have
\[\frac{\p x_j}{\p x^k}=\frac{\p^2\phi}{\p x^j\p x^k}=g_{jk},\]
hence
\[\ud x_1\wedge\ud x_2\wedge\ud x_3=\det g\cdot\ud x^1\wedge\ud x^2\wedge\ud x^3.\]
On the other hand, as $\varphi\wedge\hat\varphi=|\varphi|^2\dfrac{\omega^3}{3!}$, we see that
\[|\varphi|^2=4\det g.\]
Because $\Omega$ is holomorphic, both $\varphi$ and $\hat\varphi$ are closed. It follows that $\ud(|\varphi|^2\hat\varphi)=\ud|\varphi|^2\wedge\hat\varphi$, and that
\[\begin{split}&\Lambda_\omega\ud(|\varphi|^2\hat\varphi)=\sum_j\frac{\p|\varphi|^2}{\p x^j}\iota_{\p y_j}\hat\varphi\\
=\quad&\frac{\p|\varphi|^2}{\p x^3}\ud x_1\wedge\ud x_2-\frac{\p|\varphi|^2}{\p x^2}\ud x_1\wedge\ud x_3+\frac{\p|\varphi|^2}{\p x^1}\ud x_2\wedge\ud x_3\\ &-\frac{\p|\varphi|^2}{\p x^3}\ud y_1\wedge\ud y_2+\frac{\p|\varphi|^2}{\p x^2}\ud y_1\wedge\ud y_3-\frac{\p|\varphi|^2}{\p x^1}\ud y_2\wedge\ud y_3.\end{split}\]
Consequently
\[\begin{split}&\ud\Lambda_\omega\ud(|\varphi|^2\hat\varphi)=\left(\frac{\p^2|\varphi|^2}{\p x_1\p x^1}+\frac{\p^2|\varphi|^2}{\p x_2\p x^2}+\frac{\p^2|\varphi|^2}{\p x_3\p x^3}\right)\ud x_1\wedge\ud x_2\wedge\ud x_3\\ -&\ud\left(\frac{\p|\varphi|^2}{\p x^1}\right)\wedge\ud y_2\wedge\ud y_3-\ud y_1\wedge\ud\left(\frac{\p|\varphi|^2}{\p x^2}\right)\wedge\ud y_3-\ud y_1\wedge\ud y_2\wedge\ud\left(\frac{\p|\varphi|^2}{\p x^3}\right).\end{split}\]
Therefore, to prove the theorem, we only need to show that
\[\begin{split}\p_t(\ud x_1\wedge\ud x_2\wedge\ud x_3)=&\frac{1}{16}\left(\frac{\p^2|\varphi|^2}{\p x_1\p x^1}+\frac{\p^2|\varphi|^2}{\p x_2\p x^2}+\frac{\p^2|\varphi|^2}{\p x_3\p x^3}\right)\ud x_1\wedge\ud x_2\wedge\ud x_3,\\
\p_t(\ud x_j)=&\frac{1}{16}\ud\left(\frac{\p|\varphi|^2}{\p x^j}\right)\textrm{ for }j=1,2,3.\end{split}\]
For the first identity above, we notice that
\[\frac{1}{16}\sum_j\frac{\p}{\p x_j}\left(\frac{\p|\varphi|^2}{\p x^j}\right)=\frac{1}{4}\sum_{j,k}\frac{\p x^k}{\p x_j}\frac{\p^2}{\p x^k\p x^j}\det g=\sum_{j,k}g^{jk}\p_t g_{jk}=\frac{\p_t\det g}{\det g}.\]
Therefore
\[\begin{split}&\p_t(\ud x_1\wedge\ud x_2\wedge\ud x_3)=(\p_t\det g)\ud x^1\wedge\ud x^2\wedge\ud x^3\\
=~&\frac{\det g}{16}\left(\frac{\p^2|\varphi|^2}{\p x_1\p x^1}+\frac{\p^2|\varphi|^2}{\p x_2\p x^2}+\frac{\p^2|\varphi|^2}{\p x_3\p x^3}\right)\ud x^1\wedge\ud x^2\wedge\ud x^3\\
=~&\frac{1}{16}\left(\frac{\p^2|\varphi|^2}{\p x_1\p x^1}+\frac{\p^2|\varphi|^2}{\p x_2\p x^2}+\frac{\p^2|\varphi|^2}{\p x_3\p x^3}\right)\ud x_1\wedge\ud x_2\wedge\ud x_3.\end{split}\]
From the evolution equation of $g$, we know that there is an affine function $l$ such that
\[\p_t\phi=\frac{1}{4}\det g+l,\]
hence
\[\p_t(\ud x_j)=\ud\left(\p_t\frac{\p\phi}{\p x^j}\right)=\ud\left(\frac{\p}{\p x^j}\left(\p_t\phi\right)\right)=\frac{1}{4}\ud\left(\frac{\p\det g}{\p x^j}\right)=\frac{1}{16}\ud\left(\frac{\p|\varphi|^2}{\p x^j}\right).\]
\end{proof}

\begin{rmk}
In \cite{fei2021e}, the dual Anomaly flow was defined as a flow of Hermitian metrics on the dual manifold. Thus the setting for both the original flow and its dual were in complex geometry, and the duality only holds modulo lower order terms.
From this point of view, the symplectic and the geometric settings as described above are more naturally dual, as the duality is now exact.
\end{rmk}


\begin{rmk}
In the above theorem, we made the choice of the phase angle so that $\varphi$ and $\hat\varphi$ have the expressions we worked with. In fact, it is well legitimate to choose $\varphi'=\hat\varphi$ and $\hat\varphi'=-\varphi$ instead. By similar and straightforward computation, one can show that $\varphi'$ also satisfies the Type IIA flow
\[\p_t\varphi'=\frac{1}{16}\ud\Lambda_\omega\ud(|\varphi'|^2\hat\varphi'),\]
therefore Theorem \ref{mirror} holds for arbitrary choice of the phase angle.
\end{rmk}

In Theorem \ref{mirror}, we showed that the Type IIA flow and Type IIB flow are related to each other by T-duality in the semi-flat limit. Following this idea, once we have a geometric flow in the complex setting, by applying T-duality, one may arrive at a natural flow in symplectic geometry. In particular, we can find the T-dual of the K\"ahler-Ricci flow in the symplectic world.

\smallskip

In \cite{fei2021e}, the semi-flat reduction of the K\"ahler-Ricci flow was shown to be given by
\begin{equation}\label{kr}
\p_t g_{jk}=\frac{1}{2}\frac{\p^2}{\p x^j\p x^k}\log\det g.
\end{equation}
A similar calculation as in Theorem \ref{mirror} yields the following theorem.

\begin{thm}\label{riccidual}~\\
Under the K\"ahler-Ricci flow (\ref{kr}), the associated 3-form $\varphi$ satisfies the flow
\begin{equation}
\p_t\varphi=\frac{1}{2}\ud\Lambda_{\omega}\ud(\log|\varphi|^2\cdot\hat\varphi).
\end{equation}
\end{thm}

By dropping the non-essential factor of $\dfrac{1}{2}$, we shall call the flow
\begin{equation}\label{dr}
\begin{cases}
&\p_t\varphi=\ud\Lambda_{\omega}\ud(\log|\varphi|^2\cdot\hat\varphi)\\
&\varphi_{t=0}=\varphi_0\textrm{ is a closed, primitive, and positive $3$-form}
\end{cases}
\end{equation}
the \emph{dual Ricci flow}.

Due to the presence of the factor $\log|\varphi|^2$ in ($\ref{dr}$), one can imagine that the short-time existence and uniqueness of (\ref{dr}) would be very challenging to establish, if there is such a theorem.

\section{The Hitchin gradient flow on a symplectic manifold}

We begin by recalling the functional introduced by Hitchin \cite{hitchin2000}. Let $M$ be an oriented compact 6-manifold. Following \cite{hitchin2000, fei2015b}, given any positive 3-form $\varphi$ on $M$, or equivalently any reduction of the structure group of $M$ to $\mathrm{SL}(3,\ccc)$, there is a naturally associated almost complex structure $J_\varphi$, and another 3-form $\hat\varphi=J_\varphi\varphi$, such that the form $\Omega=\varphi+\sqrt{-1}\hat\varphi$ is a nowhere vanishing (3,0)-form with respect to the almost complex structure $J_\varphi$.

The Hitchin functional is defined as
\[H(\varphi)=\frac{1}{2}\int_M\varphi\wedge\hat\varphi.\]

Assume now that $\varphi$ is closed. Hitchin proposed the variational problem of finding the critical points of $H(\varphi)$, subject to the constraint of $\varphi$ being in a given de Rham cohomology class. In particular, he showed in \cite{hitchin2000} that
\[\delta H=\int_M\delta\varphi\wedge\hat\varphi.\]
Hence the critical points of $H(\varphi)$ are exactly those $\varphi$ such that $\ud\hat\varphi=0$, or equivalently, such that $J_\varphi$ is integrable.

\smallskip
We would like to approach this problem by considering a gradient flow for the Hitchin functional. For this, we need to introduce a metric on $M$.
A natural way to do so is to put a symplectic form $\omega$ on $M$. As shown in \cite{fei2020b}, the almost complex structure $J_\varphi$ is compatible with $\omega$ if and only if $\varphi$ is primitive with respect to $\omega$. In this way, we can rewrite Hitchin's variational formula as
\[\delta H=\int_M(\delta\varphi,\varphi)\frac{\omega^3}{3!}.\]
Since the cohomology class of $\varphi$ is not changing, we may write $\delta\varphi=\ud\delta\beta$, so
\[\delta H=\int_M\langle\ud\delta\beta,\varphi\rangle\frac{\omega^3}{3!}=\int_M\langle\delta\beta,\ud^\dagger\varphi\rangle\frac{\omega^3}{3!}.\]
Consequently the gradient flow of Hitchin's functional is $\p_t\beta=\ud^\dagger\varphi$, or
\[\p_t\varphi=\p_t\ud\beta=\ud\ud^\dagger\varphi.\]
In this form, the gradient flow of the Hitchin functional on a compact $6$-dimensional symplectic manifold can be viewed as the $6$-dimensional version of Bryant's Laplacian flow on $7$-manifolds, whose stationary points are given by manifolds with $G_2$-holonomy \cite{bryant2006}.

\medskip
Our first observation is that the symplectic version of the gradient flow for the Hitchin functional is actually a degenerate version
of the Type IIA flow introduced in \cite{fei2020b}:

\begin{thm}
\label{hi:thm}~\\
Let $M$ be a compact $6$-dimensional manifold equipped with a symplectic form $\omega$. Then the gradient flow of the Hitchin functional with a closed, primitive, and positive initial $3$-form can be equivalently expressed as
\bea
\label{hi}
\p_t\varphi=\ud\Lambda_{\omega}\ud\hat\varphi
\eea
\end{thm}

\begin{proof} Note that, formally, the flow (\ref{hi}) preserves the closedness and primitiveness of $\varphi$. It suffices to prove the following identity
\[\ud\ud^\dagger\varphi=\ud\Lambda_\omega\ud\hat\varphi.\]
for any closed primitive positive $3$-form $\varphi$. Since $\ud^\dagger\varphi=-\!*\!\ud\!*\!\varphi=-\!*\!\ud\hat\varphi$, we only need to show that $-\!*\!\ud\hat\varphi=\Lambda_\omega\ud\hat\varphi$. But it follows from \cite[Lemma 18]{fei2020b} that $\ud\hat\varphi$ is the product of $\omega$ with a primitive $(1,1)$-form.
\end{proof}

We observe that, just as the gradient flow of the $6$-dimensional Hitchin functional is related to the Type IIA flow by Theorem \ref{hi:thm}, there is also a 7-dimensional version of Hitchin's functional, whose gradient flow is exactly Bryant's Laplacian flow  for closed $\mathrm{G}_2$-structures, as explained in \cite{lotay2020}.

\bigskip
We show next that the symplectic flow (\ref{hi}) is weakly parabolic. For this we need to determine the eigenvalues of its principal symbol. We follow the procedure developed in \cite{fei2020b} for the Type IIA flow. We start from the following formula derived in \cite{fei2020b}
\bea
\label{eq:variation}
\delta\hat\varphi
=-J_\varphi(\delta\varphi)-2{\langle\delta\varphi,\hat\varphi\rangle\over |\varphi|^2}\varphi+2{\langle\delta\varphi,\varphi\rangle\over|\varphi|^2}\hat\varphi.
\eea
Since the symbol of the exterior differential is $\xi\wedge{\cdot}$, it follows that the symbol of the operator $\varphi\to \ud\Lambda_\omega\ud(\hat\varphi)$ is given by
\bea
\delta\varphi
\ \to\
\xi \wedge\bigg\{\Lambda_\omega\left[\xi\wedge (-J_\varphi \delta\varphi-2{\langle\delta\varphi,\hat\varphi\rangle\over |\varphi|^2}\varphi+2{\langle\delta\varphi,\varphi\rangle\over|\varphi|^2}\hat\varphi)\right]\bigg\}.
\eea

It is no loss of generality to assume that
\bea
\xi=e^1,\qquad |\varphi|=1,
\eea
to write $J_\varphi=J$, $\Lambda_\omega=\Lambda$, and to work on an adapted frame for $\varphi$, where we have, as shown in \cite{fei2020b}
\[\begin{split}
&
\o=e^{12}+e^{34}+e^{56}\nonumber\\
&
\varphi={1\over 2}(e^{135}-e^{146}-e^{245}-e^{236}),\quad
\hat\varphi
={1\over 2}(e^{136}+e^{145}+e^{235}-e^{246}).
\end{split}\]
Note that in the frame $\{e_j\}$ of vectors, we have $Je_{2k-1}=e_{2k}$, $J(e_{2k})=-e_{2k-1}$, so in this co-frame, we have $J(e^{2k-1})=-e^{2k}$,
$J(e^{2k})=e^{2k-1}$. We need to identify the eigenvalues of this operator, restricted to the subspace $W$ of $3$-forms $\delta\varphi$ satisfying the constraints resulting from $\varphi$ being closed and primitive
\bea
W=\{\delta\varphi;\quad  \xi\wedge \delta\varphi=0, \quad \Lambda(\delta\varphi)=0\},
\eea
which can be worked out to be
\bea
W=e^1\wedge W'
\eea
where $W'$ is the space of $2$-forms  $\gamma$ on the vector space $V'$ spanned by $\{e_j\}_{j=3}^6$, which satisfies the constraint $\Lambda'\gamma=0$, where $\Lambda'$ is the Hodge contraction operator on $V'$ with respect to the symplectic form $\o'=e^{34}+e^{56}$. Next we compute
\[\begin{split}
&
\Lambda[e^1\wedge (J\delta\varphi)]
=
\Lambda[e^1\wedge (-e^2\wedge J\gamma)]
=
-J\gamma
\nonumber\\
&
\Lambda[e^1\wedge\hat\varphi]={1\over 2}(e^{35}-e^{46})
\nonumber\\
&
\Lambda[e^1\wedge\varphi]=-{1\over 2}(e^{45}+e^{36}).
\end{split}\]
We also work out the inner products $\langle\delta\varphi,\varphi\rangle$ and $\langle\delta\varphi,\hat\varphi\rangle$,
\[\begin{split}
&
\langle\delta\varphi,\varphi\rangle={1\over 2}\langle e^1\wedge\gamma, e^{135}-e^{146}\rangle={1\over 2}\langle\gamma,e^{35}-e^{46}\rangle
\nonumber\\
&
\langle\delta\varphi,\hat\varphi\rangle={1\over 2}\langle e^1\wedge\gamma,e^{136}+e^{145} \rangle={1\over 2}\langle\gamma, e^{36}+e^{45}\rangle.
\end{split}\]
Thus the symbol map on $W'$ becomes
\bea
\gamma\
\to
\
J\gamma+\langle\delta\varphi,\hat\varphi\rangle
(e^{45}+e^{36})
+
\langle\delta\varphi,\varphi\rangle(e^{35}-e^{46})
\eea
It is convenient to use the following basis for the $5$-dimensional space $W'$,
\bea
\kappa=e^{34}-e^{56},
\ \mu_1^\pm=e^{45}\pm e^{36},
\  \mu_2^\pm=e^{35}\pm e^{46}
\eea
which are all eigenvectors of $J$,
\[\begin{split}
&
J(e^{34}-e^{56})=(e^{34}-e^{56}) \nonumber\\
&
J(e^{45}+e^{36})=-(e^{45}+e^{36}),
\quad
J(e^{45}-e^{36})=e^{45}-e^{36} \nonumber\\
&
J(e^{35}-e^{46})=-(e^{35}-e^{46}),
\quad
J(e^{35}+e^{46})=e^{35}+e^{46}.
\end{split}\]
The symbol becomes the following operator
\bea
\gamma
\ \to
\ J\gamma+{1\over 2}\<\gamma,\mu_2^-\>\mu_2^-+{1\over 2}\<\gamma,\mu_1^+\>\mu_1^+
\eea
This implies readily

\begin{lemma}
The above basis turns out to be all eigenvectors of the symbol map
\[\begin{split}
&
\kappa\ \to \kappa,
\quad
\mu_1^+\ \to \ 0,
\quad
\mu_2^+\ \to\ \mu_2^+\\
&
\mu_1^-\ \to\ \mu_1^-,
\quad
\mu_2^-\ \to\ 0.
\end{split}\]
\end{lemma}

Thus the Hitchin gradient flow is more degenerate than the Type IIA flow, whose principal symbol only has one zero eigenvalue. This additional degeneracy prevents a proof of short-time existence for general symplectic manifolds and general data along the lines of \cite{fei2020b}. Nevertheless, as we shall see in the next section, the Hitchin gradient flow exists on many interesting manifolds and exhibits a variety of remarkable phenomena. For the remaining part of this section, we shall just assume that the Hitchin gradient flow exists on a symplectic manifold, and derive the corresponding evolution equations for geometric quantities of interest, such as the metric $g_{ij}$ and $|\varphi|^2$.

\medskip
The key point in this derivation is that the Type IIA structure is preserved under the flow (\ref{hi}), hence we are free to use various identities in Type IIA geometry developed in \cite{fei2020b, fei2020c}.
Otherwise, in \cite{fei2020b} and \cite{fei2020c}, two different methods were given for deriving the evolution equations for the metric and the term $|\varphi|^2$ in the Type IIA flow. Both methods can be readily adapted to the present case of the Hitchin gradient flow. However, since the Hitchin gradient flow is given by a Laplacian, and the method of \cite{fei2020c} gives explicit Bochner-Kodaira formulas for the Laplacian on $3$-forms, it is easiest to just extract from \cite{fei2020c} what we need.

\smallskip

Recall that the metric $g_\varphi$ is defined by $g_\varphi(X,Y)=\omega(X,J_\varphi Y)$. We denote it by just $g_{ij}$ for simplicity, and also consider the metric $\tilde g_{ij}$ defined by
\[
\tilde g_{ij}=|\varphi|^2 g_{ij}=\tilde g_{ij}=-\varphi_{jkp}\varphi_{iab}\o^{ka}\o^{pb}.
\]

We consider first the flow of the metric $\tilde g_{ij}$, which is then given by
\[\begin{split}
\p_t\tilde g_{ij}&=-\p_t\varphi_{jkp}\varphi_{iab}\o^{ka}\o^{pb}+(i\leftrightarrow j)\\
&=-(\ud\ud^\dagger \varphi)_{jkp}\varphi_{iab}\o^{ka}\o^{pb}+(i\leftrightarrow j)
\end{split}
\]
The contribution of $\ud\ud^\dagger$ has been worked out in \cite{fei2020c}, Lemma 14. It is given by
\[\begin{split}
&
-(\ud\ud^\dagger\varphi)_{jkp}\varphi_{iab}\o^{ka}\o^{pb}+(i\leftrightarrow j)
\nonumber\\
&
\quad
=
|\varphi|^2\big\{
Rg_{ij}-2({\frak D}_kN_{ij}{}^k+{\frak D}_kN_{ji}{}^k)
+(-\na_\mu\na^\mu u+N^2)g_{ij}+2(N_i{}^k{}_j+N_j{}^k{}_i)\p_ku
\nonumber\\
&
\quad\quad\quad\quad
-4(N_-^2)_{ij}+8(N_+^2)_{ij}\big\}
\end{split}\]
where the scalar function $u$ is defined by
\bea
u=\log |\varphi|^2
\eea
and the quadratic expressions $N_-^2$ and $N_+^2$ in the Nijenhuis tensor are defined by
\[
(N_+^2)_{ij}=N^{pk}{}_iN_{pkj},
\quad
(N_-^2)_{ij}
=N^{kp}{}_iN_{pkj}.
\]
We can now make use of the following identities for the Nijenhuis tensor and the scalar curvature $R$ in Type IIA geometry established in \cite{fei2020b}
\[\begin{split}
&
(N_-^2)_{ij}=2(N_+^2)_{ij}-{1\over 4}|N|^2 g_{ij},
\\
&
R=\Delta u-|N|^2=\na_\mu\na^\mu u-|N|^2
\end{split}
\]
and arrive at the following result:

\begin{lemma}
The flow of the metric $\tilde g_{ij}$ is given by
\bea
\p_t\tilde g_{ij}
=
-|\varphi|^2\big\{2({\frak D}_kN_{ij}{}^k+{\frak D}_kN_{ji}{}^k)-
2\p_pu(N_j{}^p{}_i+N_i{}^p{}_j)-|N|^2 g_{ij}\big\}
\nonumber
\eea
Here the covariant derivatives and scalar curvature are with respect to the metric $g_{ij}$.
\end{lemma}

\smallskip
Next, we derive the flow of $u=\log |\varphi|^2$ and of $g_{ij}=|\varphi|^{-2}\tilde g_{ij}$.
We have
\[\begin{split}
\p_t\log\, {\rm det}\,\tilde g&=\tilde g^{ij}\p_t\tilde g_{ij}=|\varphi|^{-2}g^{ij}
\p_t\tilde g_{ij}
=
-(6R-6\na_\mu\na^\mu u)
=-6(R-\Delta u).
\end{split}\]
In view of the above formula relating $R$, $\Delta u$ and $|N|^2$ in Type IIA geometry, this reduces to
\bea
\p_t\log\, {\rm det}\,\tilde g=6|N|^2.
\eea
Since we also have
\[\begin{split}
\p_t\log|\varphi|^2={1\over 6}\p_t\log {\rm det}\,\tilde g
\end{split}\]
we obtain the flow for $u=\log |\varphi|^2$,
\bea
\p_t\log |\varphi|^2=|N|^2.
\eea
The flow of $g_{ij}$ readily follows
\[\begin{split}
\p_tg_{ij}&=\p_t(|\varphi|^{-2}\tilde g_{ij})
=
|\varphi|^{-2}\p_t\tilde g_{ij}-(\p_t\log|\varphi|^2)|\varphi|^{-2}\tilde g_{ij}
\nonumber\\
&=
|\varphi|^{-2}\p_t\tilde g_{ij}-(\p_t\log|\varphi|^2) g_{ij}.
\end{split}\]
We summarize the formulas in the following lemma:

\begin{lemma}
The flows of the metric $g_{ij}$ and of the scalar function $u=\log|\varphi|^2$ in the Hitchin gradient flow are given by
\[\begin{split}
&\p_tg_{ij}
=
-\big\{2({\frak D}_kN_{ij}{}^k+{\frak D}_kN_{ji}{}^k)-
2\p_pu(N_j{}^p{}_i+N_i{}^p{}_j)\big\}\\
&\p_tu=|N|^2.
\end{split}
\]
\end{lemma}

We note that the Hitchin functional certainly increases along its gradient flow. But the lemma implies a much stronger property, namely that the integrand in the functional is pointwise monotone increasing.
Now the Ricci curvature in Type IIA geometry is given by
\bea
{\frak D}_kN_{ij}{}^k+{\frak D}_kN_{ji}{}^k
=R_{ij}+2(N_-^2)_{ij}-(\na_i\na_ju)^{J}
\eea
so that the above flows can also be rewritten as, with the notation
$(\na_i\na_ju)^J=\dfrac{1}{2}(\na_i\na_ju+J_i{}^pJ_j{}^q\na_p\na_qu)$ from \cite{fei2020b},

\begin{lemma}
The flows of the metric $g_{ij}$ and of the scalar function $u=\log|\varphi|^2$ in terms of the Ricci curvature are given by
\[\begin{split}
&\p_tg_{ij}
=-\big\{2R_{ij}+4(N_-^2)_{ij}-2(\na_i\na_ju)^J-2\p_pu(N_j{}^p{}_i+N_i{}^p{}_j)\big\}
\\
&
\p_tu=|N|^2.
\end{split}
\]
\end{lemma}

There is yet another natural way of rewriting this flow, using the Ricci curvature formula
\cite[Eq. (6.53), Eq. (6.59)]{fei2020b}. Thus we get

\begin{lemma}
The flows of the metric $g_{ij}$ and of the scalar function $u=\log |\varphi|^2$ in the Hitchin gradient flow can also be written as
\begin{equation}\label{chi}
\begin{split}
&\p_t g_{ij}=-R_{ij}+R_{Ji,Jj}+2\p_pu(N_i{}^p{}_j+N_j{}^p{}_i)\\
&\p_t u=|N|^2
\end{split}
\end{equation}
\end{lemma}

\begin{rmk}
Formally the coupled system (\ref{chi}) makes sense for any compact symplectic manifold, therefore it can be viewed as a generalization of the gradient flow of Hitchin's functional (\ref{hi}) to arbitrary symplectic manifolds.
\end{rmk}
\begin{rmk}
The system (\ref{chi}) also shows that the gradient flow of the Hitchin functional can be viewed as a perturbation of the anti-complexified Ricci flow of L\^e-Wang \cite{le2001}. This is the gradient flow of the Blair-Ianus functional \cite{blair1986}, which is the $L^2$ norm of the Nijenhuis tensor $N$. Explicitly, the gradient flow can expressed as
\[\p_tg_{ij}=-R_{ij}+R_{Ji,Jj}.\]
We see that the first equation in (\ref{chi}) only differs from the anti-complexified Ricci flow by a first order term. \end{rmk}
\begin{rmk}
If we run the flow (\ref{hi}) on locally homogeneous half-flat symplectic manifolds (as in the case of \cite[Section 9.3.2]{fei2020b}), the induced flow of $g$ is exactly L\^{e}-Wang's anti-complexified Ricci flow because that $u$ is a constant in space. In particular, this flow can be used to find optimal compatible almost complex structures.
\end{rmk}

\section{Examples}

In this section, we present a few explicit examples of the gradient flow of 6D-Hitchin's functional on locally homogeneous symplectic half-flat manifolds. These manifolds have been studied in \cite{fei2020b} from the perspective of the Type IIA flow.

\begin{ex}
Let $M$ be the nilmanifold constructed by de Bartolomeis-Tomassini \cite[Example 5.2]{bartolomeis2006}. The Lie algebra of the nilpotent Lie group is characterized by invariant 1-forms $\{e^1,\dots,e^6\}$ satisfying
\[\begin{split}
&\ud e^1=\ud e^2=\ud e^3=\ud e^5=0,\nonumber\\
&\ud e^4=e^{15},\qquad \ud e^6=e^{13}.\nonumber
\end{split}
\]
Clearly $\omega=e^{12}+e^{34}+e^{56}$ defines an invariant symplectic structure. Moreover, this nilpotent Lie group admits co-compact lattices so all the constructions descend to compact nilmanifolds. Consider the ansatze
\bea
\label{data:nil}
\varphi=\varphi_{a,b}=(1+a)e^{135}-e^{146}-e^{245}-e^{236}+b(e^{134}-e^{156}),
\eea
it is easy to check that $\varphi_{a,b}$ is primitive and closed for any $a,b$. The positivity condition for $\varphi_{a,b}$ is that $\dfrac{1}{16}|\varphi|^4=1+a-b^2>0$. By straightforward calculations, we get
\bea
|\varphi|^2\hat\varphi=4((1+a-b^2)e^1\!\wedge\!(e^{36}+e^{45})+e^2\wedge(be^{34}+(1+a)e^{35}\!-\!e^{46}\!-\!be^{56})).\nonumber
\eea
It follows that
\[\begin{split}
&
\ud(|\varphi|^2\hat\varphi)=4e^{12}(e^{34}+2be^{35}-e^{56}),\nonumber\\
&
\Lambda_\omega\ud(|\varphi|^2\hat\varphi)=4(e^{34}+2be^{35}-e^{56}),\nonumber\\
&
\ud\Lambda_\omega\ud(|\varphi|^2\hat\varphi)=8e^{135}.\nonumber
\end{split}
\]
So the gradient flow of 6D-Hitchin's functional reduces to the ODE system
\[\begin{cases}&\dot a = \dfrac{2}{\sqrt{1+a-b^2}}\\
&\dot b=0
\end{cases},\]
which can be solved explicitly
\[\begin{cases}&(1+a-b_0^2)^{3/2}=(1+a_0-b_0^2)^{3/2}+3t,\\
&b=b_0
\end{cases}.\]
In particular the flow exists for all time and
\[|N|^2=(1+a-b^2)^{-3/2}=\frac{1}{(1+a_0-b_0^2)^{3/2}+3t}\to 0,\]
as $t$ goes to infinity.
\end{ex}

\begin{ex}
Consider the symplectic half-flat structure on the solvmanifold $M$ constructed by Tomassini and Vezzoni in \cite[Theorem 3.5]{tomassini2008}. The geometry of this solvmanifold is characterized by invariant 1-forms $\{e^j\}_{j=1}^6$ satisfying
\[\begin{split}
&
\ud e^1 = -\lambda e^{15},\quad \ud e^2= \lambda e^{25},\quad \ud e^3 = -\lambda e^{36},
\nonumber\\
&
\ud e^4= \lambda e^{46},\quad \ud e^5=0,\quad \ud e^6=0,\nonumber
\end{split}
\]
where $\lambda=\log \dfrac{3+\sqrt{5}}{2}$. One can easily check that $\omega=e^{12}+e^{34}+e^{56}$ is an invariant symplectic form on $M$. Consider the ansatze
\bea
\varphi=\alpha(e^{135}+e^{136})+\beta(e^{145}-e^{146})+\gamma(e^{235}-e^{236})-\delta(e^{245}+e^{246}).\label{solvans}
\eea
A direct calculation gives
\[\begin{split}
|\varphi|^2\hat\varphi&=8(-\alpha\beta\gamma(e^{135}-e^{136})+\alpha\beta\delta(e^{145}+e^{146})\nonumber\\
&+\alpha\gamma\delta(e^{235}+e^{236}) +\beta\gamma\delta(e^{245}-e^{246})).\nonumber
\end{split}
\]
The nondegenerate condition is that $|\varphi|^4=64\alpha\beta\gamma\delta>0$. It follows that
\[\begin{split}
\ud(|\varphi|^2\hat\varphi)&=16\lambda(\alpha\beta\gamma e^{1356}+\alpha\beta\delta e^{1456}-\alpha\gamma\delta e^{2356}+\beta\gamma\delta e^{2456}),\nonumber\\
\Lambda_\omega \ud(|\varphi|^2\hat\varphi)&=16\lambda(\alpha\beta\gamma e^{13}+\alpha\beta\delta e^{14}-\alpha\gamma\delta e^{23}+\beta\gamma\delta e^{24}),\nonumber\\
\ud\Lambda_\omega\ud(|\varphi|^2\hat\varphi)&=16\lambda^2(\alpha\beta\gamma(e^{135}+e^{136})+\alpha\beta\delta(e^{145}-e^{146})\nonumber\\
&+\alpha\gamma\delta(e^{235}-e^{236})- \beta\gamma\delta(e^{245}+e^{246})).\nonumber
\end{split}
\]
After linear time rescaling, the gradient flow of Hitchin's functional under our ansatze reduces to
\[\begin{split}
&
\dot\alpha=\frac{\alpha\beta\gamma}{\sqrt{\alpha\beta\gamma\delta}},\qquad
\dot\beta=\frac{\alpha\beta\delta}{\sqrt{\alpha\beta\gamma\delta}},\nonumber\\
&
\dot\gamma=\frac{\alpha\gamma\delta}{\sqrt{\alpha\beta\gamma\delta}},\qquad
\dot\delta=\frac{\beta\gamma\delta}{\sqrt{\alpha\beta\gamma\delta}}.\nonumber
\end{split}
\]
For simplicity, let us assume that all of $\alpha,\beta,\gamma,\delta$ are positive. It is easy to see that there exist positive constants $C_1$ and $C_2$ such that $\alpha(t)=C_1\delta(t)$ and $\beta(t)=C_2\gamma(t)$. The ODE system simplifies to
\bea
\dot\gamma=\sqrt{\frac{C_1}{C_2}}\delta,\qquad
\dot\delta=\sqrt{\frac{C_2}{C_1}}\gamma.\nonumber
\eea
Again this system can be solve explicitly as
\[\begin{split}
&\alpha=\sqrt{C_1}(Ae^t+Be^{-t}),\quad \beta=\sqrt{C_2}(Ae^t-Be^{-t}),\nonumber\\
&\gamma=\frac{1}{\sqrt{C_2}}(Ae^t-Be^{-t}),\quad \delta=\frac{1}{\sqrt{C_1}}(Ae^t+Be^{-t}),\nonumber
\end{split}
\]
where $A>0$, $B$ are constants determined by initial data. In particular the flow exists for all time and $\lim_{t\to\infty}|\varphi|^2=\infty$. However, the limit $\lim_{t\to\infty}J_t=J_\infty$ does exist. This is because
\[\begin{split}\varphi_\infty:=\lim_{t\to\infty}\frac{\varphi}{|\varphi|}=~&\sqrt{\frac{C_1}{8}}(e^{135}+e^{136})+\sqrt{\frac{C_2}{8}}(e^{145}-e^{146})\\
&+\sqrt{\frac{1}{8C_2}}(e^{235}-e^{236})-\sqrt{\frac{1}{8C_1}}(e^{245}+e^{246})\end{split}\]
exists. In fact, $J_\infty$ is a harmonic almost complex structure in the sense of \cite{le2001}, namely the Ricci curvature is $J_\infty$-invariant. In addition, $\varphi_\infty$ satisfies
\[\ud\Lambda_\omega\ud\hat\varphi_\infty=2\lambda^2\varphi_\infty.\]
This case also provides an example of convergence of anti-complexified Ricci flow to a non-K\"ahler metric.
\end{ex}

\begin{rmk}
When $|\varphi|$ is a constant over space, the Hitchin gradient flow is simply a time-rescaled version of the Type IIA flow. Therefore we can use Raffero's technique \cite{raffero2021} to produce many special solutions to the Hitchin gradient flow (\ref{hi}) as well.
\end{rmk}

\section{Additional remarks}

It may be worth considering the following $\epsilon$-regularization of the Hitchin gradient flow
\bea
\label{epsilon}
\p_t\varphi=\ud\Lambda_\omega\ud(|\varphi|^\epsilon \hat\varphi)
\eea
for each $\epsilon>0$. The same arguments for the Type IIA flow show that the flow preserves primitiveness and should be a well-defined flow of Type IIA geometries. While we do not expect that the corresponding solutions will have a limit as $\epsilon\to 0$, it is conceivable that certain important notions may have a limit. A model situation may be Landau-Ginzburg models and renormalized energies. It is also intriguing that the dual Ricci flow can be interpreted as another limit of this regularization.



\begin{thebibliography}{10}

\bibitem[BI86]{blair1986}
D.~E. Blair and S.~Ianus.
\newblock Critical associated metrics on symplectic manifolds.
\newblock In {\em Nonlinear Problems in Geometry}, volume~51 of {\em
  Contemporary Mathematics Series}, pages 23--29. AMS, 1986.

\bibitem[Bry06]{bryant2006}
R.~L. Bryant.
\newblock Some remarks on ${G}_2$-structures.
\newblock In {\em Proceedings of 12th G\"okova Geometry-Topology Conference},
  pages 75--109. International Press, 2006.

\bibitem[BS09]{brendle2009}
S.~Brendle and R.~M. Schoen.
\newblock Manifolds with 1/4-pinched curvature are space forms.
\newblock {\em Journal of the American Mathematical Society}, 22(1):287--307,
  2009.

\bibitem[dBT06]{bartolomeis2006}
P.~de~Bartolomeis and A.~Tomassini.
\newblock On the {M}aslov index of {L}agrangian submanifolds of generalized
  {C}alabi-{Y}au manifolds.
\newblock {\em International Journal of Mathematics}, 17(08):921--947, 2006.

\bibitem[Don87]{donaldson1987}
S.~K. Donaldson.
\newblock Infinite determinants, stable bundles and curvature.
\newblock {\em Duke Mathematical Journal}, 54(1):231--247, 1987.

\bibitem[ES64]{eells1964}
J.~Eells and J.H. Sampson.
\newblock Harmonic mappings of {R}iemannian manifolds.
\newblock {\em American Journal of Mathematics}, 86(1):109--160, 1964.

\bibitem[Fei15]{fei2015b}
T.~Fei.
\newblock Stable forms, vector cross products and their applications in
  geometry.
\newblock {\em arXiv: 1504.02807}, 2015.

\bibitem[FP21]{fei2021e}
T.~Fei and S.~Picard.
\newblock Anomaly flow and {T}-duality.
\newblock {\em Pure and Applied Mathematics Quarterly}, 17(3):1083--1112, 2021.

\bibitem[FPPZa]{fei2020c}
T.~Fei, D.~H. Phong, S.~Picard, and X.-W. Zhang.
\newblock Bochner-{K}odaira formulas and the {T}ype {IIA} flow.
\newblock {\em arXiv: 2012.01550}.

\bibitem[FPPZb]{fei2020}
T.~Fei, D.~H. Phong, S.~Picard, and X.-W. Zhang.
\newblock Estimates for a geometric flow for the {T}ype {IIB} string.
\newblock {\em arXiv: 2004.14529, to appear in Mathematische Annalen}.

\bibitem[FPPZc]{fei2020b}
T.~Fei, D.~H. Phong, S.~Picard, and X.-W. Zhang.
\newblock Geometric flows for the {T}ype {IIA} string.
\newblock {\em arXiv: 2011.03662, to appear in Cambridge Journal of
  Mathematics}.

\bibitem[FPPZ21]{fei2021b}
T.~Fei, D.~H. Phong, S.~Picard, and X.-W. Zhang.
\newblock Dynamical stability of the {T}ype {IIA} flow and applications in
  symplectic geometry.
\newblock {\em to appear}, 2021.

\bibitem[Ham82]{hamilton1982}
R.~S. Hamilton.
\newblock Three-manifolds with positive {R}icci curvature.
\newblock {\em Journal of Differential Geometry}, 17(2):255--306, 1982.

\bibitem[He21]{he2021}
W.-Y. He.
\newblock Evolution equations for non-degenerate 2-forms.
\newblock {\em International Mathematics Research Notices}, 2021(6):4349--4391,
  2021.

\bibitem[Hit00]{hitchin2000}
N.~J. Hitchin.
\newblock The geometry of three-forms in six dimensions.
\newblock {\em Journal of Differential Geometry}, 55(3):547--576, 2000.

\bibitem[Lot20]{lotay2020}
J.~D. Lotay.
\newblock Geometric flows of ${G}_2$ structures.
\newblock In {\em Lectures and Surveys on $G_2$-Manifolds and Related Topics},
  volume~84 of {\em Fields Institute Communications}, pages 113--140. Springer,
  2020.

\bibitem[LW01]{le2001}
H.~V. L\^e and G.-F. Wang.
\newblock Anti-complexified {R}icci flow on compact symplectic manifolds.
\newblock {\em Journal f\"ur die reine und angewandte Mathematik}, 530:17--31,
  2001.

\bibitem[Per02]{perelman2002}
G.~Y. Perelman.
\newblock The entropy formula for the {R}icci flow and its geometric
  applications.
\newblock {\em arXiv: math/0211159}, 2002.

\bibitem[Per03a]{perelman2003b}
G.~Y. Perelman.
\newblock Finite extinction time for the solutions to the {R}icci flow on
  certain three-manifolds.
\newblock {\em arXiv: math/0307245}, 2003.

\bibitem[Per03b]{perelman2003}
G.~Y. Perelman.
\newblock Ricci flow with surgery on three-manifolds.
\newblock {\em arXiv: math/0303109}, 2003.

\bibitem[Pho20]{phong2020}
D.~H. Phong.
\newblock Geometric partial differential equations from unified string
  theories.
\newblock arXiv:1906.03693, in {\em Proceedings of the International Consortium of Chinese
  Mathematicians, 2018: Second Annual Meeting}, pages 67--87. International
  Press, 2020.

\bibitem[PPZ18a]{phong2018e}
D.~H. Phong, S.~Picard, and X.-W. Zhang.
\newblock Anomaly flows.
\newblock {\em Communications in Analysis and Geometry}, 26(4):955--1008, 2018.

\bibitem[PPZ18b]{phong2018b}
D.~H. Phong, S.~Picard, and X.-W. Zhang.
\newblock Geometric flows and {S}trominger systems.
\newblock {\em Mathematische Zeitschrift}, 288(1-2):101--113, 2018.

\bibitem[Raf21]{raffero2021}
A.~Raffero.
\newblock Special solutions to the {T}ype {IIA} flow.
\newblock {\em arXiv: 2107.12264}, 2021.

\bibitem[ST14]{streets2014}
J.~Streets and G.~Tian.
\newblock Symplectic curvature flow.
\newblock {\em Journal f\"ur die reine und angewandte Mathematik},
  2014(696):143--185, 2014.

\bibitem[TV08]{tomassini2008}
A.~Tomassini and L.~Vezzoni.
\newblock On symplectic half-flat manifolds.
\newblock {\em Manuscripta Mathematica}, 125(4):515--530, 2008.

\bibitem[TY14]{tseng2014}
L.-S. Tseng and S.-T Yau.
\newblock Generalized cohomologies and supersymmetry.
\newblock {\em Communications in Mathematical Physics}, 326(3):875--885, 2014.

\end{thebibliography}

\end{document}